\documentclass[a4paper,11pt]{article}
\usepackage[latin1]{inputenc}
\usepackage[english]{babel}
\usepackage{amsmath}
\usepackage{amsfonts}
\usepackage{amssymb}
\usepackage{epsfig}
\usepackage{amsopn}
\usepackage{amsthm}
\usepackage{color}
\usepackage{graphicx}
\usepackage{subfigure}
\usepackage{enumerate}
\newtheorem{theorem}{Theorem}[section]

\newtheorem{lemma}[theorem]{Lemma}
\newtheorem{proposition}[theorem]{Proposition}
\newtheorem{definition}[theorem]{Definition}

\newtheorem*{theorem*}{Theorem}
\newtheorem*{lemma*}{Lemma}
\newtheorem*{remark*}{Remark}
\newtheorem*{definition*}{Definition}
\newtheorem*{proposition*}{Proposition}
\newtheorem*{corollary*}{Corollary}
\numberwithin{equation}{section}
\newcommand{\real}{\mathbb{R}}

\let\ced=\c         

\def\qed{\,\unskip\kern 6pt \penalty 500
\raise -2pt\hbox{\vrule \vbox to8pt{\hrule width 6pt
\vfill\hrule}\vrule}\par}
\definecolor{darkblue}{rgb}{0.05, .05, .65}
\definecolor{darkgreen}{rgb}{0.1, .65, .1}
\definecolor{darkred}{rgb}{0.8,0,0}
\newcommand{\beqn}{\begin{equation}}
\newcommand{\eeqn}{\end{equation}}
\newcommand{\bear}{\begin{eqnarray}}
\newcommand{\eear}{\end{eqnarray}}
\newcommand{\bean}{\begin{eqnarray*}}
\newcommand{\eean}{\end{eqnarray*}}
\begin{document}

\title{\huge \bf A special self-similar solution and existence of global solutions for a reaction-diffusion equation with Hardy potential}

\author{
\Large Razvan Gabriel Iagar\,\footnote{Departamento de Matem\'{a}tica
Aplicada, Ciencia e Ingenieria de los Materiales y Tecnologia
Electr\'onica, Universidad Rey Juan Carlos, M\'{o}stoles,
28933, Madrid, Spain, \textit{e-mail:} razvan.iagar@urjc.es},\\
[4pt] \Large Ariel S\'{a}nchez,\footnote{Departamento de Matem\'{a}tica
Aplicada, Ciencia e Ingenieria de los Materiales y Tecnologia
Electr\'onica, Universidad Rey Juan Carlos, M\'{o}stoles,
28933, Madrid, Spain, \textit{e-mail:} ariel.sanchez@urjc.es}\\
[4pt] }
\date{}
\maketitle

\begin{abstract}
Existence and uniqueness of a specific self-similar solution is established for the following reaction-diffusion equation with Hardy singular potential
$$
\partial_tu=\Delta u^m+|x|^{-2}u^p, \qquad (x,t)\in \real^N\times(0,\infty),
$$
in the range of exponents $1\leq p<m$ and dimension $N\geq3$. The self-similar solution is unbounded at $x=0$ and has a logarithmic vertical asymptote, but it remains bounded at any $x\neq0$ and $t\in(0,\infty)$ and it is a weak solution in $L^1$ sense, which moreover satisfies $u(t)\in L^p(\real^N)$ for any $t>0$ and $p\in[1,\infty)$. As an application of this self-similar solution, it is shown that there exists at least a weak solution to the Cauchy problem associated to the previous equation for any bounded, nonnegative and compactly supported initial condition $u_0$, contrasting with previous results in literature for the critical limit $p=m$.
\end{abstract}

\


\smallskip


\section{Introduction}

The goal of this paper is to establish some properties related, first, to self-similar solutions and second, to general solutions to the Cauchy problem associated to the following reaction-diffusion equation involving a Hardy potential
\begin{equation}\label{eq1}
\partial_tu=\Delta u^m+|x|^{-2}u^p, \qquad (x,t)\in\real^N\times(0,\infty),
\end{equation}
posed in the range of exponents $1\leq p<m$ and dimension $N\geq3$ (as usual when considering Hardy potentials). In the study of the general Cauchy problem, we consider the class of bounded, non-negative and compactly supported initial conditions
\begin{equation}\label{init.cond}
u(x,0)=u_0(x), \qquad u_0\in L^{\infty}(\real^N), \qquad {\rm supp}\,u_0\subseteq B(0,R), \qquad u_0\geq0, \ u_0\not\equiv0.
\end{equation}
for some $R>0$. The definition of a solution to the Cauchy problem \eqref{eq1}-\eqref{init.cond} will be made precise in the text.

The study of parabolic reaction-diffusion equations involving a singular potential, that is, a spatially-dependent coefficient of the form $|x|^{-\sigma}$ with $\sigma>0$ at the reaction term, was brought into attention by the classical paper by Baras and Goldstein \cite{BG84}, in which the problem of existence of solutions to the following linear equation
\begin{equation}\label{eq2}
\partial_tu=\Delta u+K|x|^{-2}u,
\end{equation}
with $K>0$ and posed in a bounded domain and with homogeneous boundary conditions, is considered. The most striking result of this well-known work is that existence depends strongly on the constant $K>0$, thus for a wide class of non-negative initial conditions, a weak solution to Eq. \eqref{eq2} exists if and only if $K\leq (N-2)^2/4$, which is the optimal constant in Hardy's inequality. More precisely, for larger constants $K>(N-2)^2/4$ there are no solutions except for the trivial one, as it is proved that any solution will present \emph{complete instantaneous blow-up}, which means that it becomes infinite at any $x\in\real^N$ and at any time $t>0$. Cabr\'e and Martel \cite{CM99} considered general potentials $a(x)\in L^1_{\rm loc}(\Omega)$ instead of $|x|^{-2}$ (where $\Omega\subset\real^N$ is a bounded domain) and raised the same problem of the threshold between existence and non-existence of solutions to Eq. \eqref{eq2}. Their study gives a sharp condition for the non-existence of solutions in terms of the form of the spectrum of the (linear) operator $-\Delta-a(x)$. The connections between the Hardy inequality and the properties of solutions to the heat equation with a Hardy potential were later developed in \cite{VZ00, VZ12}, making use of more refined functional inequalities. Another extension of the study performed for Eq. \eqref{eq2} to more general linear equations is the outcome of the paper by Goldstein and Zhang \cite{GZ02}, where the problem of existence and non-existence of solutions is addressed for linear parabolic operators with the Laplacian replaced by an operator with variable coefficients in the leading order.

The development of the theory for parabolic equations with singular potentials saw recently a great amount of papers dealing with the semilinear version of Eq. \eqref{eq1}, that is, letting in the latter $m=1$ and any $p>1$. More precisely, a number works considered the question of existence of solutions for the semilinear equation
\begin{equation}\label{eq.semi}
u_t=\Delta u+K|x|^{\sigma}u^p, \qquad p>1, \ -2<\sigma<0,
\end{equation}
in suitable functional spaces, with weakly regular initial conditions (see works such as \cite{BSTW17, BS19, CIT21a, CIT21b, T20}) or with singular data \cite{HT21}. The same question of existence of global solutions is studied with fractional diffusion in \cite{HS21, HT21}. Deeper properties of the dynamics of Eq. \eqref{eq.semi}, such as similarity solutions and behavior near blow-up for general solutions, have been considered first by Filippas and Tertikas \cite{FT00}. In this work, the authors proved that for $\sigma\in(-2,0)$ there exists a unique, decreasing blow-up self-similar solution to Eq. \eqref{eq.semi} provided that $1<p<p_S=(N+2+2\sigma)/(N-2)$ and then an infinity of such self-similar solutions for $p_S<p<p_{JL}$, where $p_{JL}$ is a larger critical exponent (with a complicated explicit expression that we omit here) known as the Joseph-Lundgren exponent. Their analysis has been extended and completed by Mukai and Seki \cite{MS21} with a study of the range $p>p_{JL}$, where the so-called \emph{Type II blow-up} occurs and in which exact blow-up rates and asymptotic expansions in suitable regions of the space are obtained. Since Eq. \eqref{eq.semi} still allows to use techniques of convolution with the heat kernel and representation formulas for solutions, many techniques employed in all these works are not suitable for quasilinear equations.

Entering the quasilinear world, we start our description of precedents from papers such as \cite{Qi98, Su02}, which consider always $p>m$ and a singular potential $|x|^{\sigma}$ with $\sigma>-2$ (which could be either negative or positive) and establish the Fujita-type exponent $p_F(\sigma)=m+(2+\sigma)/N$ by analogy with to the famous Fujita work \cite{Fu66}. The super-critical fast diffusion $(N-2)/N<m<1$ is also considered. It is thus proved in \cite{Qi98} that, for $\max\{1,m\}<p\leq p_F(\sigma)$ all the non-trivial solutions blow up in finite time, while for $p>p_F(\sigma)$ global in time solutions in self-similar form are constructed, while \cite{Su02} studies finite time blow-up and global existence in dependence on the spatial decay as $|x|\to\infty$ of the initial condition when $p>p_F(\sigma)$. Later on, different quasilinear diffusion operators were considered together with Hardy-type potentials: fast diffusion $0<m<1$ in \cite{GK03, GGK05}, $p$-Laplacian diffusion in \cite{AP00} and doubly nonlinear diffusion in \cite{Ko04}.

A remarkable feature of the equations involving Hardy-type potentials (such as the ones introduced above) is that we do not see only a competition between the diffusion and the reaction terms for influencing their dynamics, but also a \emph{second competition between regions} with $|x|$ small and regions with $|x|$ large, due to the presence of a singular weight strongly acting close to $x=0$ but having a very low influence far-away. An example of the effects of such a competition has been given recently by the authors for Eq. \eqref{eq1} in the borderline case $p=m$ and with the Hardy potential $K|x|^{-2}$, $0<K\leq(N-2)^2/4$, in the short note \cite{IS20b}. By means of a transformation to the porous medium equation, an interesting case of continuation after blow-up is found: all the solutions blow up (either instantaneously or in finite time) only at $x=0$, but they keep belonging to the functional spaces allowing for the development of the weak theory for any $t>0$. The optimal Hardy constant $K=(N-2)^2/4$ limits the existence and non-existence ranges in this specific quasilinear case treated in \cite{IS20b} in the same way as it did for the linear case in \cite{BG84}. In another recent contribution \cite{IMS21b}, a unique self-similar solution presenting grow-up as $t\to\infty$ but not finite time blow-up is constructed for the equation
\begin{equation}\label{eq1bis}
\partial_tu=\Delta u^m+|x|^{\sigma}u^p, \qquad (x,t)\in\real^N\times(0,\infty),
\end{equation}
in the range of exponents $-2<\sigma<0$ and $1\leq p<1-\sigma(m-1)/2<m$, but the limiting case $\sigma=-2$ is not considered there. We may thus say that the current paper completes the theory established in these previous works, but with a different outcome if compared to them.

\medskip

\noindent \textbf{Main results.} Our first goal is to construct a specific self-similar solution to Eq. \eqref{eq1}. Recalling the more general form \eqref{eq1bis} (where $\sigma=-2$ in the case of Eq. \eqref{eq1}), we have seen in previous works that an essential number for the expected type and behavior of the self-similar solution (according to \cite{IS21a, IMS22, IMS21b}) is
$$
L=\sigma(m-1)+2(p-1)=2(p-m)<0,
$$
in our range of exponents $p\in[1,m)$, thus we expect again to have self-similar solutions in forward form (if any) with grow-up as $t\to\infty$, that is, in the general form $u(x,t)=t^{\alpha}f(|x|t^{-\beta})$ for some $\alpha>0$, $\beta>0$ to be determined. These self-similarity exponents (whose general form is given in, for example, \cite{IMS21b}) become for $\sigma=-2$
$$
\alpha=-\frac{\sigma+2}{\sigma(m-1)+2(p-1)}=0, \qquad \beta=-\frac{m-p}{\sigma(m-1)+2(p-1)}=\frac{1}{2}.
$$
It follows that self-similar solutions to Eq. \eqref{eq1} (if existing) are expected to have the form
\begin{equation}\label{SSS}
u(x,t)=f(\xi), \qquad \xi=|x|t^{-1/2},
\end{equation}
where the self-similar profile $f(\xi)$ is a solution to the following differential equation
\begin{equation}\label{SSODE}
(f^m)''(\xi)+\frac{N-1}{\xi}(f^m)'(\xi)+\frac{1}{2}\xi f'(\xi)+\xi^{-2}f(\xi)^p=0.
\end{equation}
We are now in a position to state our first main result, concerning the uniqueness of a specific self-similar solution with compact support to Eq. \eqref{eq1}.
\begin{theorem}\label{th.uniqSS}
Let $m>1$, $p\in[1,m)$ and $N\geq3$. There exists a unique compactly supported self-similar solution (in the sense of Definition \ref{def.sol} below) of the form \eqref{SSS} to Eq. \eqref{eq1}. Its self-similar profile is supported on an interval $[0,\xi_0]$ with $\xi_0\in(0,\infty)$, and has the following local behavior at its endpoints
\begin{equation}\label{beh.Q1}
f(\xi)\sim\left[-\frac{m-p}{m(N-2)}\ln\,\xi+K\right]^{1/(m-p)}, \qquad {\rm as} \ \xi\to0, \ \xi>0,
\end{equation}
for some $K\in\real$, and
\begin{equation}\label{beh.P1}
f(\xi)\sim\left[\frac{m-1}{4m}(\xi_0^2-\xi^2)\right]_{+}^{1/(m-1)}, \qquad {\rm as} \ \xi\to\xi_0\in(0,\infty), \ \xi<\xi_0.
\end{equation}
\end{theorem}

We notice that the self-similar solution given by Theorem \ref{th.uniqSS} is not a ``standard" solution at $x=0$, since it has a vertical asymptote at $\xi=0$, as it follows from the local behavior \eqref{beh.Q1}, but which is of logarithmic type, thus belonging to any $L^p$ space for $p\in[1,\infty)$. We thus need to introduce a formulation of the notion of solution to Eq. \eqref{eq1}. This is the content of the next definition, denoting in it and in the sequel by $u(t)$ the mapping $x\mapsto u(x,t)$ for a fixed $t\geq0$. We recall here that sometimes this type of solution is referred as \emph{very weak solution}, since we pass all the Laplacian to the test function.
\begin{definition}\label{def.sol}
By a \emph{weak solution to Eq. \eqref{eq1}} we understand a function $u\in C((0,T):L^{1}(\real^N))$ for some $T>0$, which moreover satisfies the following assumptions:
\begin{itemize}
\item $$
u(t), \ u^m(t)\in L^{1}(\real^N), \qquad \frac{u^p(t)}{|x|^2}\in L^1(\real^N), \qquad {\rm for \ any } \ t\in(0,T).
$$
\item $u$ is a solution in the sense of distributions to Eq. \eqref{eq1}, that means that for any $\varphi\in C_0^{2,1}(\real^N\times(0,T))$ and for any $t_1$, $t_2\in(0,T)$ with $t_1<t_2$ we have
\begin{equation}\label{weak}
\begin{split}
\int_{\real^N}u(t_2)\varphi(t_2)\,dx&-\int_{\real^N}u(t_1)\varphi(t_1)\,dx-\int_{t_1}^{t_2}\int_{\real^N}u(t)\varphi_t(t)\,dx\,dt\\
&-\int_{t_1}^{t_2}\int_{\real^N}u^m(t)\Delta\varphi(t)\,dx\,dt=\int_{t_1}^{t_2}\int_{\real^N}\frac{u^p(t)}{|x|^2}\varphi(t)\,dx\,dt.
\end{split}
\end{equation}
\end{itemize}
We say that a function $u\in C([0,T):L^{1}(\real^N))$ for some $T>0$ is a \emph{weak solution to the Cauchy problem} \eqref{eq1}-\eqref{init.cond} if $u$ is a weak solution to Eq. \eqref{eq1} and the initial condition is taken in $L^1$ sense, that is $u(t)\to u_0$ as $t\to0$ with convergence in $L^1(\real^N)$.
\end{definition}
The functional framework in Definition \ref{def.sol} is inspired by the one introduced in \cite[Section 6.2 and Problem 6.2]{VPME} for the standard porous medium equation, but with the adaptation required to cope with the reaction term including a singular potential at $x=0$. We emphasize here that with this notion of solution, we refrain from entering the rather technical problem of convergence of the gradients of the approximating family of solutions in Section \ref{sec.wp}. Of course, the self-similar solution in Theorem \ref{th.uniqSS} is unbounded at $x=0$ but it stays bounded at any $t>0$ and at any $x$ with $|x|>0$, as it is readily seen from \eqref{SSS}, and it is easily shown that it fulfills the conditions of Definition \ref{def.sol}. Such a situation of solutions having a single blow-up point at the origin but remaining a weak solution to the equation (in suitable Lebesgue or Sobolev spaces) at any $t>0$ has been also met in the limiting case $p=m$ to Eq. \eqref{eq1} in \cite{IS20b}. Let us finally mention here that, as a by-product of the techniques we use, we classify in the body of this paper \emph{all the possible local behaviors} of self-similar profiles solving Eq. \eqref{SSODE}.

Using the self-similar solution obtained in Theorem \ref{th.uniqSS} in the form of a ``friendly giant", that is, a very large supersolution which allows for bounds from above and dominated convergence in approximating processes (see \cite[Section 5.9]{VPME} for the classical notion of ``friendly giant" for the porous medium equation), we can establish existence of at least a weak solution for any bounded and compactly supported initial condition.
\begin{theorem}\label{th.exist}
Let $m>1$, $p\in[1,m)$ and $N\geq3$ and let $u_0$ be a function satisfying \eqref{init.cond}. Then the Cauchy problem \eqref{eq1}-\eqref{init.cond} has at least a weak solution in the sense of Definition \ref{def.sol}.
\end{theorem}
As a consequence of this theorem, we obtain that there is no constant limiting between existence and non-existence of solutions. More precisely

\medskip

\noindent \textbf{Remark.} Let $m>1$, $p\in[1,m)$ and $N\geq3$ and let $u_0$ be a function satisfying \eqref{init.cond}. For any $K>0$, the Cauchy problem associated to the equation
\begin{equation}\label{eq1.K}
\partial_tu=\Delta u^m+K|x|^{-2}u^p, \qquad (x,t)\in\real^N\times(0,\infty),
\end{equation}
with initial condition $u(x,0)=u_0(x)$ for any $x\in\real^N$, admits at least a weak solution in the sense of Definition \ref{def.sol}. This fact follows readily by employing a scaling argument. Indeed, by setting
\begin{equation}\label{scale.K}
u(x,t)=\lambda^{1/(m-1)}v(x,\lambda t), \qquad s=\lambda t, \ \lambda=K^{(m-1)/(m-p)}
\end{equation}
we find that $v$ introduced in \eqref{scale.K} solves Eq. \eqref{eq1} with an initial condition $v_0(x)=K^{-1/(m-p)}u_0(x)$, which also satisfies \eqref{init.cond}. The existence Theorem \ref{th.exist} can be then applied and we get the desired result by undoing the rescaling \eqref{scale.K}.

\medskip

This is a striking difference with the well-known result of instantaneous complete blow-up established by Baras and Goldstein \cite{BG84} for the linear case $m=p=1$ when $K>(N-2)^2/4$ and also noticed for the other limiting case $p=m>1$ in \cite{IS20b}. Indeed, the fact that $p\neq m$ is fundamental for removing this upper bound on the constant $K$ for existence of weak solutions, since it allows for a rescaling of Eq. \eqref{eq1.K} which makes it equivalent to Eq. \eqref{eq1}, that is with $K=1$, while such a rescaling is not available when $p=m$. As for the proof of Theorem \ref{th.exist}, it is a constructive one based on an approximation with a family of solutions to regular problems, and the ``friendly giant" coming from Theorem \ref{th.uniqSS} will be decisive in ensuring dominated convergence of the approximating family.

\medskip

\noindent \textbf{Structure of the paper}. The proof of Theorem \ref{th.uniqSS} is based on a change of variable transforming Eq. \eqref{SSODE} into a quadratic autonomous dynamical system and then allowing to employ the technique of a phase-space analysis to study the orbits connecting critical points, which will be equivalent to self-similar profiles with specified local behavior. Such a technique had been used recently with success by the authors in classifying self-similar solutions to reaction-diffusion equations, see for example works such as \cite{IS19, IS21a, IMS22, IS22}. We thus devote two different sections for the preliminary analysis, a Section \ref{sec.local} concerning the local analysis of the finite critical points, followed by Section \ref{sec.infty} devoted to the critical points at infinity, where the claimed logarithmic local behavior at $\xi=0$ is identified. Existence and uniqueness of self-similar solutions comes next and is split again into two Sections \ref{sec.exist} and \ref{sec.uniq}. Finally, a single Section \ref{sec.wp} is dedicated to the construction of the approximating solutions and the proof of Theorem \ref{th.exist}. The paper ends with a short section where open problems are raised and new developments of the theory of Eq. \eqref{eq1} are suggested.

\section{The dynamical system. Local analysis}\label{sec.local}

We begin here our analysis of the differential equation \eqref{SSODE} solved by the self-similar profiles to Eq. \eqref{eq1}. The following easy result will be very useful in the forthcoming analysis.
\begin{lemma}\label{lem.max}
A self-similar profile to Eq. \eqref{eq1} cannot have positive local minima. In particular, any profile $f(\xi)$ decreasing for a short interval $\xi\in(0,\delta)$ will decrease on all its positivity set.
\end{lemma}
\begin{proof}
Assume for contradiction that $\xi_1\in(0,\infty)$ is a local minimum point. We infer from evaluating Eq. \eqref{SSODE} at $\xi=\xi_1$ that
$$
(f^m)''(\xi_1)=-\xi_1^{-2}f(\xi_1)^p<0,
$$
which is a contradiction with the standard properties of a minimum point.
\end{proof}
We introduce now the following change of variable
\begin{equation}\label{PS.change}
X(\xi)=m\xi^{-2}f(\xi)^{m-1}, \qquad Y(\xi)=\frac{m}{\xi}f(\xi)^{m-2}f'(\xi), \qquad Z(\xi)=\xi^{-2}f(\xi)^{p-1},
\end{equation}
together with the new independent variable $\eta$ defined via the differential equation
$$
\frac{d\eta}{d\xi}=\frac{1}{m}\xi f(\xi)^{1-m}.
$$
This change of variable converts Eq. \eqref{SSODE} into the following quadratic autonomous system
\begin{equation}\label{PSsyst1}
\left\{\begin{array}{ll}\dot{X}=X[(m-1)Y-2X],\\
\dot{Y}=-Y^2-\frac{1}{2}Y-NXY-XZ,\\
\dot{Z}=Z[(p-1)Y-2X],\end{array}\right.
\end{equation}
where the dot derivatives are taken with respect to the independent variable $\eta$. Moreover, we observe that $X\geq0$, $Z\geq0$ and the coordinate planes $\{X=0\}$ and $\{Z=0\}$ are invariant for the system \eqref{PSsyst1}. Another simple but useful remark is that the flow of the system \eqref{PSsyst1} over the plane $\{Y=0\}$ has always negative direction. The system \eqref{PSsyst1} has the following finite critical points
$$
P_0=(0,0,0), \ P_1=\left(0,-\frac{1}{2},0\right), \ P^{\gamma}=(0,0,\gamma), \ {\rm for} \ \gamma>0,
$$
which will be analyzed below.
\begin{lemma}[Local analysis near $P_0$]\label{lem.P0}
The critical point $P_0$ behaves like an attractor for orbits coming from the half-space $\{X>0\}$ of the phase space associated to the system \eqref{PSsyst1}. The orbits entering it contain profiles with the local behavior
\begin{equation}\label{beh.P0}
\lim\limits_{\xi\to\infty}f(\xi)=K>0,
\end{equation}
for any constant $K>0$.
\end{lemma}
\begin{proof}
The linearization of the system \eqref{PSsyst1} in a neighborhood of $P_0$ has the matrix
$$
M(P_0)=\left(
         \begin{array}{ccc}
           0 & 0 & 0 \\
           0 & -\frac{1}{2} & 0 \\
           0 & 0 & 0 \\
         \end{array}
       \right),
$$
thus we have a two-dimensional center manifold and a one-dimensional stable manifold. According to \cite[Lemma 1, Section 2.4]{Carr}, all the orbits entering or going out of the critical point except for the trivial one are tangent to the center manifold. The Center Manifold Theorem \cite[Theorem 1, Section 2.12]{Pe} together with the approximation theorem \cite[Theorem 3, Section 2.5]{Carr} ensure that the center manifold near $P_0$ is well approximated by a quadratic expansion of the form
$$
Y=h(X,Z)=aX^2+bXZ+cZ^2+O(|(X,Z)|^3),
$$
with coefficients $a$, $b$, $c$ to be determined from the equation of the center manifold. Straightforward calculations lead to $a=c=0$ and $b=-2$, thus
\begin{equation}\label{interm1}
h(X,Z)=-2XZ+XO(|(X,Z)|)^2,
\end{equation}
where the fact that all the terms in the center manifold are a multiple of $X$ follows readily from the equation for the center manifold. According thus to the reduction theorem \cite[Theorem 2, Section 2.4]{Carr}, the flow on the center manifold is given by the reduced system obtained by replacing $Y$ by $h(X,Z)$ in the equations for $\dot{X}$ and $\dot{Z}$ in \eqref{PSsyst1}, namely
\begin{equation}\label{interm2}
\left\{\begin{array}{ll}\dot{X}&=-2X^2+X^2O(|(X,Z)|),\\
\dot{Z}&=-2XZ+XO(|(X,Z)|^2),\end{array}\right.
\end{equation}
thus all the orbits enter the critical point $P_0$ on the center manifold. The local behavior of the profiles contained in these orbits is obtained in a first approximation by an integration in the reduced system \eqref{interm2}, leading to $X\sim KZ$ for any constant $K>0$, which readily gives $f(\xi)\sim K$ after undoing the change of variable \eqref{PS.change}. Moreover, since on these orbits $X\to0$ and $Z\to0$, we infer from \eqref{PS.change} and the fact that $f(\xi)\sim K$ that such local behavior is taken as $\xi\to\infty$, as claimed.
\end{proof}
\begin{lemma}[Local analysis near $P_1$]\label{lem.P1}
The system \eqref{PSsyst1} has in a neighborhood of the critical point $P_1$ a two-dimensional stable manifold and a one-dimensional unstable manifold. The orbits entering $P_1$ on the stable manifold contain profiles with an interface behavior at some point $\xi_0\in(0,\infty)$ given by \eqref{beh.P1}.
\end{lemma}
\begin{proof}
The linearization of the system \eqref{PSsyst1} in a neighborhood of $P_1$ has the matrix
$$
M(P_1)=\left(
         \begin{array}{ccc}
           -\frac{m-1}{2} & 0 & 0 \\
           \frac{N}{2} & \frac{1}{2} & 0 \\
           0 & 0 & -\frac{p-1}{2} \\
         \end{array}
       \right),
$$
thus the two-dimensional stable manifold and the one-dimensional unstable manifold are obvious. The local behavior near $P_1$ is given by the fact that $Y\to-1/2$ on the orbits entering $P_1$, together with the fact that $X\to0$ and $Z\to0$. If this behavior would be taken as $\xi\to\infty$, then by writing $Y(\xi)=m\xi^{-1}(f^{m-1})'(\xi)/(m-1)$, the fact that
$$
\xi X'(\xi)=-2X(\xi)+(m-1)Y(\xi)
$$
together with an application of \cite[Lemma 2.9]{IL13} for the function $X(\xi)$ would imply that there exists a sequence $\xi_k\to\infty$ such that $Y(\xi_k)\to0$ and a contradiction to the fact that $Y(\xi)\to-1/2$. We thus deduce that the local behavior is taken, in terms of profiles, as $\xi\to\xi_0\in(0,\infty)$ from the left, which gives first that $f(\xi_0)=0$ and then
$$
(f^{m-1})'(\xi)\sim\frac{(m-1)\xi}{2m}, \qquad {\rm as} \ \xi\to\xi_0,
$$
whence the local behavior given by \eqref{beh.P1} follows by integration on a generic interval $(\xi,\xi_0)$.
\end{proof}
We are left with the critical half-line containing the points $P^{\gamma}$ with $\gamma>0$. We shall see that such points cannot contain any orbit of interest for us.
\begin{lemma}[Local analysis near $P^{\gamma}$]\label{lem.Pgamma}
There are no interesting profiles $f(\xi)$ contained in orbits of the system \eqref{PSsyst1} connecting to or from any of the points $P^{\gamma}$ with $\gamma>0$.
\end{lemma}
The proof is rather tedious but follows absolutely analogous steps as in \cite[Lemma 2.3]{IMS21b}, the fundamental aspect of the non-existence of orbits entering these points being the fact that $\sigma(m-1)+2(p-1)<0$, which holds true in our case as $\sigma=-2$ and $p<m$. We thus omit the calculations and refer the interested reader to imitate step by step the proof of \cite[Lemma 2.3]{IMS21b} given with all the details in the reference, following exactly the same changes of variable as there to reach the conclusion.

\medskip

\noindent \textbf{Differences for $p=1$.} In the case $p=1$ a difference with respect to the analysis of the critical point $P_1$ appears, since this single critical point is replaced for $p=1$ by the critical line
$$
P_1^{\gamma}=\left(0,-\frac{\beta}{\alpha},\gamma\right), \qquad \gamma>0.
$$
\begin{lemma}[Analysis of the points $P_1^{\gamma}$ for $p=1$]\label{lem.P1p1}
For any $\gamma>0$, the critical point $P_1^{\gamma}$ has a one-dimensional stable manifold, a one-dimensional unstable manifold and a one-dimensional center manifold. The orbits entering $P_1^{\gamma}$ on the stable manifold contain profiles with interface behaving as in \eqref{beh.P1}, while both the center manifold and the unstable manifold are contained completely in the invariant plane $\{X=0\}$.
\end{lemma}
The proof is easy as the local behavior of the profiles is obtained in the same way as in Lemma \ref{lem.P1} since it still holds true that $Y\to-1/2$ and $X\to0$ on these orbits. See also \cite[Lemma 2.2]{IS19} for more details.

\section{Local analysis of the critical points at infinity}\label{sec.infty}

This section is devoted to the local analysis of the critical points of the system \eqref{PSsyst1} at the infinity of the space. This analysis is needed in order to complete the panorama of all possible behaviors of the profiles $f(\xi)$ solutions to Eq. \eqref{SSODE}. We pass to the Poincar\'e hypersphere through the new variables $(\overline{X},\overline{Y},\overline{Z},W)$ defined as
$$
X=\frac{\overline{X}}{W}, \ Y=\frac{\overline{Y}}{W}, \ Z=\frac{\overline{Z}}{W}
$$
and we infer from standard theory \cite[Theorem 4, Section 3.10]{Pe} that the critical points at space infinity lie on the equator of the hypersphere, thus at points
$(\overline{X},\overline{Y},\overline{Z},0)$ where $\overline{X}^2+\overline{Y}^2+\overline{Z}^2=1$ and the following system is satisfied:
\begin{equation}\label{Poincare}
\left\{\begin{array}{ll}\overline{X}[\overline{X}\overline{Z}+(N-2)\overline{X}\overline{Y}+m\overline{Y}^2]=0,\\
(p-m)\overline{X}\overline{Y}\overline{Z}=0,\\
\overline{Z}[p\overline{Y}^2+(N-2)\overline{X}\overline{Y}+\overline{X}\overline{Z}]=0,\end{array}\right.
\end{equation}
Taking into account that we are considering only points with coordinates $\overline{X}\geq0$ and $\overline{Z}\geq0$ and that we are working in dimension $N\geq3$, we find the following critical points on the Poincar\'e hypersphere:
\begin{equation*}
\begin{split}
&Q_1=(1,0,0,0), \ \ Q_{2,3}=(0,\pm1,0,0), \ \ Q_4=(0,0,1,0), \\
&Q_5=\left(\frac{m}{\sqrt{(N-2)^2+m^2}},-\frac{N-2}{\sqrt{(N-2)^2+m^2}},0,0\right).
\end{split}
\end{equation*}
The analysis of the critical point $Q_1$ will be of utmost importance for the rest of the paper. Both critical points $Q_1$ and $Q_5$ can be analyzed by projecting on the $X$ variable according to the theory in \cite[Section 3.10]{Pe}. More precisely, \cite[Theorem 5(a), Section 3.10]{Pe} ensures that, if we let
\begin{equation}\label{change.inftyX}
y=\frac{Y}{X}, \qquad z=\frac{Z}{X}, \qquad w=\frac{1}{X},
\end{equation}
the critical points at infinity $Q_1$ and $Q_5$ can be identified with the critical points $(0,0,0)$, respectively $(-(N-2)/m,0,0)$ of the following system
\begin{equation}\label{systinf1}
\left\{\begin{array}{ll}\dot{y}=-(N-2)y-z-my^2-\frac{1}{2}yw,\\
\dot{z}=-(m-p)yz,\\
\dot{w}=2w-(m-1)yw,\end{array}\right.
\end{equation}
obtained by plugging \eqref{change.inftyX} into the equations of the system \eqref{PSsyst1}.
\begin{lemma}[Local analysis near $Q_1$]\label{lem.Q1}
The critical point $Q_1$ presents a two-dimensional center-unstable manifold and a one-dimensional stable manifold. The orbits going out on the center-unstable manifold contain profiles with the local behavior \eqref{beh.Q1} as $\xi\to0$, while the stable manifold is contained in the invariant plane $\{w=0\}$.
\end{lemma}
\begin{proof}
The linearization of the system \eqref{systinf1} in a neighborhood of $Q_1=(0,0,0)$ has the matrix
$$
M(Q_1)=\left(
         \begin{array}{ccc}
           -(N-2) & -1 & 0 \\
           0 & 0 & 0 \\
           0 & 0 & 2 \\
         \end{array}
       \right).
$$
We set $g:=(N-2)y+z$ in order to put the system \eqref{systinf1} into the canonical form for the center manifold theorem, that is
\begin{equation*}
\left\{\begin{array}{ll}\dot{g}=-(N-2)g-(g-z)\left[\frac{m}{N-2}g-\frac{p}{N-2}z+\frac{1}{2}w\right],\\
\dot{z}=-\frac{m-p}{N-2}z(g-z),\\
\dot{w}=2w-\frac{m-1}{N-2}w(g-z),\end{array}\right.
\end{equation*}
and the center manifold has the form $(g(z),w(z))$. We easily infer by setting $g(z)=az^2+O(z^3)$, $w(z)=bz^2+O(z^3)$ and applying the approximation theorem \cite[Theorem 3, Section 2.5]{Carr} that $a=-p/(N-2)^2$ and $b=0$ and thus $g(z)=-pz^2/(N-2)^2+O(z^3)$, that is, $(N-2)y+z=-pz^2/(N-2)^2+O(z^3)$ and the flow on any (possible not unique) center manifold is given by the reduced equation
$$
\dot{z}=\frac{m-p}{N-2}z^2.
$$
We thus find that on the center manifolds we have an unstable flow, going out of $Q_1$. This is linked with the one-dimensional unstable manifold to generate a center-unstable manifold, whose orbits behave in a sufficiently small neighborhood of $Q_1$ as the solutions of the approximating system
\begin{equation}\label{interm3}
\left\{\begin{array}{ll}\dot{z}=\frac{m-p}{N-2}z^2,\\
\dot{w}=2w+\frac{m-1}{N-2}zw.\end{array}\right.
\end{equation}
The system \eqref{interm3} can be integrated to find that
\begin{equation}\label{trajec}
w\sim Cz^{(m-1)/(m-p)}e^{-2(N-2)/(m-p)z}, \qquad C\in(0,\infty),
\end{equation}
which is equivalent in terms of profiles to
$$
\frac{1}{m}\xi^2f(\xi)^{1-m}\sim C\left[\frac{f(\xi)^{p-m}}{m}\right]^{(m-1)/(m-p)}\exp\left[-\frac{2m(N-2)}{m-p}f(\xi)^{m-p}\right].
$$
We perform straightforward calculations in the right hand side of the previous equivalence to get that
\begin{equation}\label{interm4}
f(\xi)^{m-p}\sim K-\frac{m-p}{m(N-2)}\ln\,\xi,
\end{equation}
where $K\in\real$ is an arbitrary constant. Moreover, since $z=Z/X\to0$ at $Q_1$, we get that $f(\xi)^{p-m}\to0$, that is, $f(\xi)\to\infty$. This implies that necessarily the equivalence \eqref{interm4} is taken in the limit as $\xi\to0$ and the claimed local behavior \eqref{beh.Q1} follows.
\end{proof}
\begin{lemma}[Local analysis near $Q_5$]\label{lem.Q5}
The critical point $Q_5$ is an unstable node. The orbits going out of it contain profiles with the local behavior
\begin{equation}\label{beh.Q5}
f(\xi)\sim D\xi^{-(N-2)/m}, \qquad {\rm as} \ \xi\to0,
\end{equation}
with $D>0$ arbitrary.
\end{lemma}
\begin{proof}
The linearization of the system \eqref{systinf1} near $Q_5=(-(N-2)/m,0,0)$ has the matrix
$$
M(Q_5)=\left(
         \begin{array}{ccc}
           N-2 & -1 & \frac{N-2}{2m} \\
           0 & \frac{(m-p)(N-2)}{m} & 0 \\
           0 & 0 & 2+\frac{(m-1)(N-2)}{m} \\
         \end{array}
       \right),
$$
hence $Q_5$ is an unstable node, since we are in dimension $N\geq3$. The orbits going out of it satisfy $y\to-(N-2)/m$, that is
\begin{equation}\label{interm5}
\frac{\xi f'(\xi)}{f(\xi)}\to-\frac{N-2}{m},
\end{equation}
while also $z\to0$, that is, $Z/X=f(\xi)^{p-m}\to0$, which leads to $f(\xi)\to\infty$. A simple argument by contradiction shows that if these limits are taken as $\xi\to\xi_0\in(0,\infty)$, then an integration in \eqref{interm5} would contradict the fact that $f(\xi)\to\infty$ as $\xi\to\xi_0$. It thus follows that all the previous limits are taken as $\xi\to0$ and by integration we find the claimed behavior \eqref{beh.Q5}.
\end{proof}
The critical points $Q_2$ and $Q_3$ are analyzed through the change of variable
$$
x=\frac{X}{Y}, \qquad z=\frac{Z}{Y}, \qquad w=\frac{1}{Y},
$$
according to \cite[Theorem 5(b), Section 3.10]{Pe}, leading to the system
\begin{equation}\label{systinf2}
\left\{\begin{array}{ll}\pm\dot{x}=-mx-(N-2)x^2-\frac{1}{2}xw-x^2z,\\
\pm\dot{z}=-pz-\frac{1}{2}zw-(N-2)xz-xz^2,\\
\pm\dot{w}=-w-\frac{1}{2}w^2-Nxw-xzw,\end{array}\right.
\end{equation}
where the signs have to be chosen according to the direction of the flow. We notice that with respect to $Q_2$ one has to choose the minus sign in \eqref{systinf2}, while when analyzing the flow near $Q_3$ one has to choose the plus sign, since $\dot{Y}$ is negative near both $Q_2$ and $Q_3$ but the direction of the flow is reversed. We thus identify $Q_2$ with the origin of \eqref{systinf2} when taken the minus sign, which is an unstable node, and $Q_3$ with the origin of \eqref{systinf2} when taken the plus sign, which is a stable node.
\begin{lemma}[Local analysis near $Q_2$ and $Q_3$]\label{lem.Q23}
The orbits going out of $Q_2$ to the finite part of the phase space contain profiles $f(\xi)$ which change sign at some $\xi_0\in(0,\infty)$ in the sense that $f(\xi_0)=0$, $(f^m)'(\xi_0)>0$. The orbits entering the point $Q_3$ from the finite part of the phase space contain profiles $f(\xi)$ which change sign at some $\xi_0\in(0,\infty)$ in the sense that $f(\xi_0)=0$, $(f^m)'(\xi_0)<0$.
\end{lemma}
The proof follows the same lines as in similar lemmas in previous works such as, for example, \cite[Lemma 2.6]{IS21a} for $1<p<m$ and \cite[Lemma 2.7]{IS19} for $p=1$ to which we refer the reader. We are thus left with the critical point $Q_4$, whose analysis employing \cite[Theorem 5(c), Section 3.10]{Pe} is not easy to perform, since we are left with a critical point having only zero eigenvalues after the corresponding change of variable. We thus work directly with Eq. \eqref{SSODE} in order to establish that no orbit either enters from or goes out into the finite part of the phase space associated to the system \eqref{PSsyst1}.
\begin{lemma}\label{lem.Q4}
There are no profiles $f(\xi)$ solutions to Eq. \eqref{SSODE} contained in orbits entering or going out of $Q_4$ from or to the finite part of the phase space.
\end{lemma}
\begin{proof}
Assume for contradiction that there exist such profiles. An orbit connecting to $Q_4$ satisfies the conditions $Z\to\infty$, $X/Z\to0$, $Y/Z\to0$, which translated in terms of profiles imply
\begin{equation}\label{interm7}
\xi^{-2}f(\xi)^{p-1}\to\infty, \ \ f(\xi)^{m-p}\to0, \ \ \xi f(\xi)^{m-p-1}f'(\xi)\to0,
\end{equation}
the limits in \eqref{interm7} being taken either as $\xi\to\infty$, or as $\xi\to\xi_0\in(0,\infty)$, or as $\xi\to0$. Notice first that the second limit in \eqref{interm7} and the fact that $m>p$ already give that $f(\xi)\to0$, which, together with the first limit in \eqref{interm7} and the fact that $p\geq1$, immediately rule out the possibility that the limits are taken as $\xi\to\infty$ or as $\xi\to\xi_0\in(0,\infty)$. It remains to show that the limits in \eqref{interm7} cannot hold true as $\xi\to0$. But this follows the proof in \cite[Lemma 3.5, Step 3]{IMS21b} by simply letting therein $\sigma=-2$ without any further change.
\end{proof}
We are now ready to move towards the global analysis of the system \eqref{PSsyst1} and prove Theorem \ref{th.uniqSS}.

\section{Proof of Theorem \ref{th.uniqSS}: existence part}\label{sec.exist}

This section is dedicated to the proof of the existence of a self-similar profile with the local behavior given in the statement of Theorem \ref{th.uniqSS}. The proof borrows ideas from \cite[Section 4]{IMS21b} (thus we will skip a few details given in the quoted reference) and is based on a shooting technique on the two-dimensional center-unstable manifold near $Q_1$ (according to Lemma \ref{lem.Q1}) performed in the system \eqref{systinf1}. Notice that the center-unstable manifold is tangent to the vector space spanned by the eigenvector $e_3=(0,0,1)$ corresponding to the eigenvalue $\lambda_3=2$ of the matrix $M(Q_1)$ and the direction $e_2=(-1,N-2,0)$ of any center manifold (which might not be unique), which according to Lemma \ref{lem.Q1} are tangent to the direction $(N-2)y+z=0$. These vectors belong to the invariant planes $\{z=0\}$, respectively $\{w=0\}$ and the trajectories go out tangent to the one-parameter family of curves given by \eqref{trajec} for $C\in(0,\infty)$. We thus have to explore the global behavior of the limiting orbits belonging to the invariant planes.
\begin{proposition}\label{prop.shootw0}
There exists $C_*>0$ such that the orbits going out of $Q_1$ tangent to the curves \eqref{trajec} with $C\in(0,C_*)$ enter the critical point $Q_3$.
\end{proposition}
\begin{proof}
We study the limit orbit with $w=0$ by working on the reduced system obtained from \eqref{systinf1} in the invariant plane $\{w=0\}$, that is
\begin{equation}\label{systw=0}
\left\{\begin{array}{ll}\dot{y}=-(N-2)y-my^2-z,\\ \dot{z}=-(m-p)yz.\end{array}\right.
\end{equation}
The orbit going out of $Q_1$ goes out tangent to the direction $e_2=(-1,N-2)$, thus it enters the half-plane $\{y<0\}$. Since the flow of the system \eqref{systw=0} on the line $\{y=0\}$ is negative, it follows that the orbit stays forever in the negative half-plane. Consider now the isocline
\begin{equation}\label{iso1}
-(N-2)y-my^2-z=0,
\end{equation}
and notice that the orbit goes out tangent to it. But since in the half-plane $\{y<0\}$ we have always $\dot{z}>0$ and the orbit starts decreasingly in $y$ and increasingly in $z$ near $Q_1=(0,0)$, we infer that
$$
\frac{dy}{dz}=\frac{(N-2)y+z+my^2}{(m-p)yz}<0
$$
in a small neighborhood of $Q_1$, hence the orbit enters the region $\mathcal{R}:=\{y<0, z>-(N-2)y-my^2\}$. Furthermore, since the direction of the flow of the system \eqref{systw=0} on the isocline \eqref{iso1} is given by the sign of $(m-p)yz<0$, it follows that our orbit going out of $Q_1$ will remain forever in the region $\mathcal{R}$. Thus, along it we always have $\dot{y}<0$, $\dot{z}>0$, which means that $y$ and $z$ have a limit along the trajectory. We prove that $z\to\infty$ and $y\to-\infty$ by removing all the other possibilities. Indeed, if this was not the case, our orbit would be in one of the following three situations:

$\bullet$ $y\to y_0\in(-\infty,0)$, $z\to z_0\in(0,\infty)$. Then $(y_0,z_0)$ must be a finite critical point of the system \eqref{systw=0} and there is no such point.

$\bullet$ $y\to y_0\in(-\infty,0)$ and $z\to\infty$. It follows that
$$
\frac{dy}{dz}=\frac{(N-2)y+my^2+z}{(m-p)yz}\to\frac{1}{(m-p)y_0}<0
$$
and this contradicts the existence of a vertical asymptote of the trajectory at $y=y_0$ since in such case we would have $dy/dz\to-\infty$ at least on some subsequence $y_k\to y_0$.

$\bullet$ $y\to-\infty$ and $z\to z_0\in(0,\infty)$. We have for $y$ very large in absolute value and $z$ close to $z_0$ that
$$
\frac{dy}{dz}\sim\frac{my}{(m-p)z_0}, \qquad {\rm whence} \qquad y\sim Ke^{mz/(m-p)z_0}\to Ke^{(m-p)/m}
$$
and again we reach a contradiction.

We thus conclude that $y\to-\infty$ and $z\to\infty$ along this trajectory, thus for $y$ and $z$ very large in absolute value
$$
\frac{dy}{dz}\sim\frac{m}{m-p}\frac{y}{z}+\frac{1}{(m-p)y},
$$
which leads to the following first approximation of the trajectory:
\begin{equation}\label{interm15}
y^2\sim Kz^{2m/(m-p)}-\frac{2z}{m+p}\sim Kz^{2m/(m-p)}, \qquad K>0.
\end{equation}
where for the latter equivalence we took into account that $2m/(m-p)>2$. Thus $y\sim-\overline{K}z^{m/(m-p)}$ for $y$, $z$ sufficiently large in absolute value and some $\overline{K}>0$, which gives $y/z\sim-\overline{K}z^{p/(m-p)}\to-\infty$, or in terms of the initial variables $Y$ and $Z$, we also get $Y/Z\to-\infty$. We also notice that $y\to-\infty$ translates into $Y/X\to-\infty$ in the variables given by \eqref{PS.change}, thus the orbit has to enter a critical point characterized by
$$
\frac{Y}{X}\to-\infty, \qquad \frac{Y}{Z}\to-\infty
$$
in a neighborhood of it, and this point is $Q_3$ according to the classification given in Section \ref{sec.infty}. Moreover, since $Q_3$ is a stable node according to Lemma \ref{lem.Q23} and the orbit going out of $Q_1$ included in the plane $\{w=0\}$ corresponds to the parameter $C=0$ in \eqref{trajec}, we reach the conclusion by standard continuity arguments similar to the ones employed in, for example, \cite[Proposition 3.4]{IS21a} or \cite[Proposition 3.3]{IS21b}.
\end{proof}
\begin{proposition}\label{prop.shootz0}
There exists $C^*>0$ such that the orbits going out of $Q_1$ tangent to the curves \eqref{trajec} with $C\in(C^*,\infty)$ enter the critical point $P_0$.
\end{proposition}
\begin{proof}
We now analyze the limit orbit going out of $Q_1$ and contained in the invariant plane $\{z=0\}$, which is the unique orbit going out of the critical (saddle) point $Q_1=(0,0)$ of the reduced system obtained from \eqref{systinf1} by letting $z=0$, that is
\begin{equation}\label{systz=0}
\left\{\begin{array}{ll}\dot{y}=-(N-2)y-my^2-\frac{1}{2}yw,\\ \dot{w}=w(2-(m-1)y).\end{array}\right.
\end{equation}
This orbit goes out of $Q_1$ tangent to $e_3=(0,1)$ and it has to stay on the invariant line $\{y=0\}$. In fact, one can easily see that the two orbits (going out and entering $Q_1=(0,0)$ in the system \eqref{systz=0}) lie on the two axis of \eqref{systz=0} since both are invariant sets. Thus this orbit will coincide with the axis $\{y=0\}$ for any $w>0$. In terms of the variables $X=1/w$, $Y$, $Z$, this shows that the limit orbit has to belong to the $X$ axis and thus trivially enters $P_0$. In order to end the proof, since $+\infty$ is not a number in order to apply a continuity argument in a neighborhood of it, we have to reverse \eqref{trajec} and write it as
$$
g(z)=\frac{1}{C}w=C_1w, \qquad g(z)=z^{(m-1)/(m-p)}e^{-2(N-2)/(m-p)z}.
$$
Noticing that $\lim\limits_{z\to0}g(z)=0$, we observe that the orbit included in the invariant plane $\{z=0\}$ lies at the limit $C_1\to0$, $C_1>0$ of the orbits in \eqref{trajec}, which corresponds to $C\to+\infty$. Since, according to Lemma \ref{lem.P0}, $P_0$ is a stable point for orbits coming from the half-space $\{X>0\}$ of the phase space associated to the system \eqref{PSsyst1}, the conclusion follows once more from a continuity argument applied for $C_1$ in a right-neighborhood of zero.
\end{proof}
Before going to the proof of the existence, we need one more preparatory result proving that at least two coordinates are monotone along all the orbits going out of $Q_1$. This is essential in order to prevent such orbits from oscillating infinitely many times.
\begin{lemma}\label{lem.noncycle}
The coordinates $X$ and $Z$ are decreasing along the orbits going out of $Q_1$.
\end{lemma}
\begin{proof}
The flow of the system \eqref{PSsyst1} over the plane $\{Y=0\}$ is given by the sign of the expression $-XZ\leq0$, thus no orbit can cross the plane $\{Y=0\}$ from left to right. Since the orbits going out of $Q_1$ enter the half-space $\{Y<0\}$, they will stay forever in this negative half-space, hence $\dot{X}<0$ and $\dot{Z}<0$ on these orbits, as claimed.
\end{proof}
We can now complete the proof of the existence of the specific self-similar profile.
\begin{proof}[Proof of Theorem \ref{th.uniqSS}: existence]
We associate (by tangency) the manifold going out of $Q_1$ with the one-parameter family of curves given by \eqref{trajec} as explained before. We then consider the following three sets
\begin{equation*}
\begin{split}
&\mathcal{A}=\{C\in(0,\infty): {\rm the \ orbit \ with \ parameter} \ C \ {\rm in \ \eqref{trajec} \ enters} \ P_0\},\\
&\mathcal{C}=\{C\in(0,\infty): {\rm the \ orbit \ with \ parameter} \ C \ {\rm in \ \eqref{trajec} \ enters} \ Q_3\},\\
&\mathcal{B}=\{C\in(0,\infty): {\rm the \ orbit \ with \ parameter} \ C \ {\rm in \ \eqref{trajec} \ does \ neither \ enter} \ P_0 \ {\rm nor} \ Q_3\}.
\end{split}
\end{equation*}
We then deduce from Propositions \ref{prop.shootw0} and \ref{prop.shootz0} and the fact that $Q_3$ is a stable node, while $P_0$ is an attractor for the orbits coming from the region $\{X>0\}$ of the phase space, that both sets $\mathcal{A}$ and $\mathcal{C}$ are open and non-empty. It then follows that $\mathcal{B}$ is closed and non-empty. Let now $C\in\mathcal{B}$. The orbit going out of $Q_1$ tangent to the curve \eqref{trajec} corresponding to this value of $C$ does not enter $P_0$, nor $Q_3$ and we infer from the monotonicity of its $X$ and $Z$ components established in Lemma \ref{lem.noncycle} that it has to enter a critical point. This can be easily seen by showing that the coordinate $Y$ also has a limit along the orbit, by considering possible sequences of maxima and minima of the $Y$ component and prove that they converge to the same value. A detailed proof of this rather standard argument follows the one of \cite[Proposition 4.10]{ILS22}. It thus have to enter the \emph{unique} remaining critical point which has a stable manifold, that is $P_1$. The profiles contained in such orbits corresponding to $C\in\mathcal{B}$ fulfill the statement of Theorem \ref{th.uniqSS}.
\end{proof}
We thus conclude that all the profiles going out of $Q_1$ are decreasing on their support and are either compactly supported or having a horizontal asymptote as $\xi\to\infty$. This fact will be used in the proof of the uniqueness of the profile with interface, performed in the next section.

\section{Proof of Theorem \ref{th.uniqSS}: uniqueness}\label{sec.uniq}

In this section we complete the proof of Theorem \ref{th.uniqSS} by establishing that the set $\mathcal{B}$ introduced in the proof of the existence part is in fact a singleton. This will be done at the level of the profiles, by showing first that the profiles contained in orbits going out of $Q_1$ remain strictly ordered with respect to the free parameter $K\in\real$ in \eqref{beh.Q1} during all their support. To this end, we use a \emph{rescaling and sliding technique} stemming from (up to our knowledge) Friedman and Kamin \cite{FK80} but used in many works such as \cite{IV10, YeYin, IMS21b} among others. The first result is the following monotonicity lemma with respect to $K$. To fix the notation, let $f(\cdot;K)$ be the profile contained in the orbit going out of $Q_1$ with local behavior given by \eqref{beh.Q1} for given $K\in\real$.
\begin{lemma}\label{lem.monot}
Let $-\infty<K_1<K_2<\infty$. Then we have $f(\xi;K_1)<f(\xi;K_2)$ for any $\xi>0$ such that $f(\xi;K_1)>0$.
\end{lemma}
\begin{proof}
Let $\xi_1\in(0,\infty]$ be such that $f(\xi;K_1)>0$ for any $\xi\in(0,\xi_1)$. Since $K_1<K_2$, we infer from \eqref{beh.Q1} that $f(\xi;K_1)<f(\xi;K_2)$ for $\xi$ in a right neighborhood of the origin. We can thus introduce
\begin{equation}\label{interm8}
\xi_{*}=\inf\{\xi\in(0,\xi_1):f(\xi;K_1)=f(\xi;K_2)\}.
\end{equation}
It follows that $f(\xi;K_1)<f(\xi;K_2)$ for any $\xi\in(0,\xi_{*})$. Assume for contradiction that $\xi_{*}<\xi_1$, that is, the two profiles cross each other before the edge of the support of the first one. We next argue as in \cite{YeYin, IMS21b} by letting first $g_1(\xi)=f(\xi;K_1)^m$, $g_2(\xi)=f(\xi;K_2)^m$, which solve the differential equation
\begin{equation}\label{interm9}
g''(\xi)+\frac{N-1}{\xi}g'(\xi)+\frac{1}{2}\xi(g^{1/m})'(\xi)+\xi^{-2}g(\xi)^{p/m}=0,
\end{equation}
and then introduce for any $\lambda\in[0,1]$ the following rescaling (which is standard for the porous medium equation):
\begin{equation}\label{resc}
f_{\lambda}(\xi)=\lambda^{-2/(m-1)}f(\lambda\xi;K_1), \qquad g_{\lambda}(\xi)=\lambda^{-2m/(m-1)}g_1(\lambda\xi).
\end{equation}
It follows from \eqref{interm9} and straightforward calculations that $g_{\lambda}$ is a solution to the differential equation
\begin{equation}\label{interm10}
g_{\lambda}''(\xi)+\frac{N-1}{\xi}g_{\lambda}'(\xi)+\frac{1}{2}\xi(g_{\lambda}^{1/m})'(\xi)+\lambda^{2(p-m)/(m-1)}\xi^{-2}g_{\lambda}(\xi)^{p/m}=0.
\end{equation}
We furthermore observe that, if $0<\lambda<\lambda'<1$, due to the monotone decreasing character of $g_1$ over $(0,\xi_1)$, which follows from Lemma \ref{lem.max}, we have $g_1(\lambda'\xi)<g_1(\lambda\xi)$ for any $\xi\in(0,\xi_1)$, and it follows readily that
$$
g_{\lambda}(\xi)=\lambda^{-2m/(m-1)}g_1(\lambda\xi)>\lambda'^{-2m/(m-1)}g_1(\lambda'\xi)=g_{\lambda'}(\xi)
$$
and also
$$
\lim\limits_{\lambda\to0}g_{\lambda}(\xi)=\lim\limits_{\lambda\to0}\lambda^{-2m/(m-1)}g_1(\lambda\xi)=+\infty
$$
uniformly on $[0,\xi_{*}]$. Finally, we infer from the local behavior \eqref{beh.Q1} that
$$
g_1(\xi)\sim\left[\frac{m-p}{m(N-2)}(-\ln\,\xi)+K_1\right]^{m/(m-p)}, \qquad {\rm as} \ \xi\to0
$$
hence
\begin{equation}\label{interm11}
g_{\lambda}(\xi)\sim\lambda^{-2m/(m-1)}\left[\frac{m-p}{m(N-2)}(-\ln\,\xi)+K_1-\frac{m-p}{m(N-2)}\ln\,\lambda\right], \qquad {\rm as} \ \xi\to0.
\end{equation}
Since $\lambda\in(0,1)$, by letting $\lambda$ sufficiently small in \eqref{interm11} such that
$$
K_1-\frac{m-p}{m(N-2)}\ln\,\lambda>K_2,
$$
we get that $g_{\lambda}(\xi)>g_2(\xi)$ for $\xi$ in a right neighborhood of the origin. All the previous arguments prove that the optimal sliding parameter
\begin{equation}\label{opt}
\lambda_0=\sup\{\lambda\in(0,1): g_2(\xi)<g_{\lambda}(\xi), \ {\rm for \ any} \ \xi\in[0,\xi_{*}]\}
\end{equation}
is correctly defined and we easily derive from \eqref{interm8} and \eqref{opt} that $\lambda_0\in(0,1)$. We furthermore infer from the optimality of $\lambda_0$ that $g_2(\xi)\leq g_{\lambda_0}(\xi)$ for any $\xi\in(0,\xi_{*}]$ and that there exists $\xi^{*}\in(0,\xi_*]$ such that $g_2(\xi^*)=g_{\lambda_0}(\xi^*)$. Assume first that $\xi^*=\xi_*$. This means that
$$
g_{\lambda_0}(\xi^*)=g_{\lambda_0}(\xi_*)=\lambda_0^{-2m/(m-1)}g_1(\lambda_0\xi_*)>g_1(\xi_*)=g_2(\xi_*)=g_2(\xi^*),
$$
which is a contradiction. Since $\xi^*>0$, we deduce that $\xi^*\in(0,\xi_*)$ and in this case the function $g_{\lambda_0}-g_2$ has a minimum point at $\xi=\xi^*$, hence
\begin{equation}\label{interm12}
g_{\lambda_0}(\xi^*)=g_2(\xi^*), \ \ g_{\lambda_0}'(\xi^*)=g_2'(\xi^*), \ \ g_{\lambda_0}''(\xi^*)\geq g_2''(\xi^*).
\end{equation}
We then deduce from Eq. \eqref{interm9} solved by $g_2$, Eq. \eqref{interm10} solved by $g_{\lambda_0}$ and the inequalities \eqref{interm12} that
\begin{equation*}
\begin{split}
0&=g_2''(\xi^*)+\frac{N-1}{\xi^*}g_2'(\xi^*)+\frac{1}{2}\xi^*(g_2^{1/m})'(\xi^*)+(\xi^{*})^{-2}g_2(\xi^*)^{p/m}\\
&\leq g_{\lambda_0}''(\xi^*)+\frac{N-1}{\xi^*}g_{\lambda_0}'(\xi)+\frac{1}{2}\xi^*(g_{\lambda_0}^{1/m})'(\xi^*)+(\xi^*)^{-2}g_{\lambda_0}(\xi^*)^{p/m}\\
&=(\xi^*)^{-2}\left[1-\lambda_0^{2(p-m)/(m-1)}\right]g_{\lambda_0}(\xi^*)^{p/m}<0,
\end{split}
\end{equation*}
since $\lambda_0\in(0,1)$ and $2(p-m)/(m-1)<0$. We thus reach again a contradiction. It follows that there cannot be any crossing point $\xi_{*}\in(0,\xi_1)$ and thus the profiles remain ordered on the support of the smallest one.
\end{proof}
We still need one more preparatory lemma before going to the proof of the uniqueness of the self-similar solution.
\begin{lemma}\label{lem.super}
Let $\lambda\in(0,1)$ and $f_{\lambda}$ be the function obtained from a profile $f(\xi)$ solution to Eq. \eqref{SSODE} with interface (that is, corresponding to a parameter $C\in\mathcal{B}$) through the rescaling \eqref{resc}. Then the function
$$
U_{\lambda}(x,t)=f_{\lambda}(|x|t^{-1/2})
$$
is a supersolution to Eq. \eqref{eq1}.
\end{lemma}
\begin{proof}
A direct calculation gives
$$
\partial_tU_{\lambda}-\Delta U_{\lambda}^m-|x|^{-2}U_{\lambda}^p=\frac{1}{t}\xi^{-2}f_{\lambda}(\xi)^{p}\left(\lambda^{2(p-m)/(m-1)}-1\right)>0,
$$
since $\lambda\in(0,1)$ and $2(p-m)/(m-1)<0$. The conclusion follows from the fact that the contact condition $(f^m)'(\xi_0)=0$ at the edge of the support $\xi_0\in(0,\infty)$ remains invariant with respect to the rescaling \eqref{resc}.
\end{proof}
The monotonicity given by Lemma \ref{lem.monot} doesn't yet give the uniqueness of the profiles with $C\in\mathcal{B}$, since there exists still the possibility of a contact exactly at the edge of the support. In order to remove this contact, we go back to the full self-similar solutions as follows.
\begin{proof}[Proof of Theorem \ref{th.uniqSS}: uniqueness]
Assume for contradiction that there are two parameters $C_1$, $C_2\in\mathcal{B}$ corresponding to two self-similar profiles $f(\cdot;K_1)$ and $f(\cdot;K_2)$ such that $-\infty<K_1<K_2<\infty$. We obtain from Lemma \ref{lem.monot} that the two profiles are totally ordered on the support of the smallest one. Letting $\xi_1$, respectively $\xi_2$ be the edges of the supports of $f(\cdot;K_1)$, respectively $f(\cdot;K_2)$, we find that $f(\xi;K_1)<f(\xi;K_2)$ for any $\xi\in(0,\xi_1]$ and that $\xi_1<\xi_2$. Consider now the rescaling \eqref{resc} and the optimal sliding parameter $\lambda_0\in(0,1)$ introduced in \eqref{opt} for the functions $g_i(\xi)=f(\xi;K_i)^m$, $i=1,2$. Since the proof of Lemma \ref{lem.monot} already gives that no contact between $g_{\lambda_0}$ and $g_2$ is possible for any $\xi\in(0,\xi_1)$, it follows that the only possible contact lies at the edge of the support $\xi=\xi_2$, hence
$$
0=f(\xi_2;K_2)=f_{\lambda_0}(\xi_2), \qquad 0<f(\xi;K_2)<f_{\lambda_0}(\xi) \ {\rm for \ any} \ \xi\in(0,\xi_2).
$$
We reintroduce the time variable and construct the functions
$$
U_2(x,t)=f(|x|t^{-1/2};K_2), \qquad U_{\lambda_0}(x,t)=f_{\lambda_0}(|x|t^{-1/2}),
$$
and notice that $U_2$ is a solution to Eq. \eqref{eq1}, while $U_{\lambda_0}$ is a supersolution to Eq. \eqref{eq1} according to Lemma \ref{lem.super}. We next \emph{separate the supports} of these functions by giving a small time delay to the bigger function in order to remove the contact at $\xi=\xi_2$ and then adjust the scaling parameter. More precisely, we start with $t=1$, where we have a precise identification of the function and the self-similar profile, to notice that for any $\delta>0$ sufficiently small we have
$$
U_2(x,1)\leq U_{\lambda_0}(x,1)<U_{\lambda_0}(x,1+\delta),
$$
which gives
$$
U_2(x,1)<\lambda_0^{-2/(m-1)}f(|x|\lambda_0(1+\delta)^{-1/2};K_1), \qquad |x|\in(0,\xi_2]
$$
and with a strict separation also in a right neighborhood of the origin as it follows from the proof of Lemma \ref{lem.monot}. We can then adjust the scaling parameter by letting some $\lambda_1\in(\lambda_0,1)$ sufficiently close to $\lambda_0$ in order to satisfy simultaneously the following conditions
\begin{equation}\label{interm13}
U_2(x,1)<\lambda_1^{-2/(m-1)}f(|x|\lambda_1(1+\delta)^{-1/2};K_1)=U_{\lambda_1}(x,1+\delta),
\end{equation}
for any $x\in\real^N$ such that $|x|\in(0,\xi_2]$, and
\begin{equation}\label{interm14}
\begin{split}
\frac{m-p}{m(N-2)}(-\ln\,|x|)+K_2&<\lambda_1^{-2(m-p)/(m-1)}\left[\frac{m-p}{m(N-2)}(-\ln\,|x|)+K_2\right.\\&\left.-\frac{m-p}{m(N-2)}\ln\,\lambda_1\right],
\end{split}
\end{equation}
for any $x\in\real^N$ such that $|x|\leq\epsilon$ for some $\epsilon>0$ fixed. The latter follows easily from the fact that $\lambda_1\in(0,1)$ and the dominant term (the one containing $\ln\,|x|$) is exactly the same in both sides. We next deduce from \eqref{interm14} that the two functions $U_2$ and $U_{\lambda_1}$ separate even more in the right neighborhood $|x|\in(0,\epsilon]$ of the origin at any later times $t>1$, since
\begin{equation*}
\begin{split}
\frac{m-p}{m(N-2)}(-\ln\,|x|)+K_2&+\frac{m-p}{2m(N-2)}\ln\,t<\lambda_1^{-2(m-p)/(m-1)}\left[\frac{m-p}{m(N-2)}(-\ln\,|x|)+K_2\right.\\&\left.-\frac{m-p}{m(N-2)}\ln\,\lambda_1\right]+
\frac{m-p}{2m(N-2)}\lambda_1^{-2(m-p)/(m-1)}\ln(t+\delta),
\end{split}
\end{equation*}
for any $t>1$ and $x\in\real^N$ such that $0<|x|\leq\epsilon$. In particular, the latter implies that $U_2(x,t)<U_{\lambda_1}(x,t+\delta)$ for any $x\in\real^N$ such that $|x|=\epsilon$ and $t\in(1,\infty)$. We then apply the comparison principle (see for example \cite{ASZ01}) in the exterior domain $(\real^N\setminus B(0,\epsilon))\times(1,\infty)$ (where the weight is uniformly bounded) together with the estimate \eqref{interm14} to conclude that $U_2(x,t)<U_{\lambda_1}(x,t+\delta)$ for any $(x,t)\in\real^N\setminus\{0\}\times(1,\infty)$. We thus find that
$$
f(|x|t^{-1/2};K_2)\leq\lambda_1^{-2/(m-1)}f(\lambda_1|x|(t+\delta)^{-1/2};K_1),
$$
for any $(x,t)\in\real^N\setminus\{0\}\times(1,\infty)$, or equivalently
$$
f(\xi;K_2)\leq\lambda_1^{-2/(m-1)}f\left(\lambda_1\xi\left(\frac{t}{t+\delta}\right)^{-1/2};K_1\right),
$$
for any $\xi\in(0,\xi_2]$ and $t>1$. By passing to the limit as $t\to\infty$ we immediately reach a contradiction with the optimality of the scaling parameter $\lambda_0$ introduced in \eqref{opt}. This contradiction implies the uniqueness of the self-similar profile with interface and the proof is complete.
\end{proof}

\section{Existence for general initial conditions}\label{sec.wp}

This section is devoted to the proof of Theorem \ref{th.exist}. The main tools in the proof are the construction of an approximating sequence of solutions to regular problems obtained by a mollification of Eq. \eqref{eq1} near $x=0$ together with a comparison with the unique self-similar solution obtained from Theorem \ref{th.uniqSS}. Let us denote by
\begin{equation}\label{uniq.SSS}
U(x,t)=f(|x|t^{-1/2}), \qquad (x,t)\in\real^N\times(0,\infty)
\end{equation}
the unique self-similar solution given by Theorem \ref{th.uniqSS}. We start with the following rather obvious preparatory result.
\begin{lemma}\label{lem.comp}
Let $u_0$ be a compactly supported function as in \eqref{init.cond}. Then there exists $\tau>0$ such that
$$
u_0(x)\leq U(x,\tau),
$$
for any $x\in\real^N$.
\end{lemma}
\begin{proof}
Let $R>0$ such that ${\rm supp}\,u_0\subseteq B(0,R)$. Since
$$
\lim\limits_{x\to0}U(x,t)=+\infty, \qquad {\rm for \ any} \ t>0
$$
it follows that there exists $R_0>0$ such that $U(x,1)=f(|x|)>\|u_0\|_{\infty}$ for any $x\in\real^N$ with $0<|x|<R_0$. Let then $\tau$ be sufficiently large such that $R\tau^{-1/2}<R_0$. Then for any $x\in\real^N$ such that $0<|x|<R$ we have $|x|\tau^{-1/2}<R_0$, hence
$$
U(x,\tau)=f(|x|\tau^{-1/2})>\|u_0\|_{\infty}\geq u_0(x),
$$
for any $x\in B(0,R)$ (where comparison is extended at $x=0$ trivially since the self-similar solution is infinite there). Since $u_0(x)=0$ for $|x|\geq R$, the conclusion follows.
\end{proof}
We are now ready to complete the proof of the existence theorem.
\begin{proof}[Proof of Theorem \ref{th.exist}]
Let $u_0$ be as in \eqref{init.cond}. For any $\epsilon>0$, consider the Cauchy problem associated to the following regularized equation near $x=0$
\begin{equation}\label{eq.eps}
\partial_tu-\Delta u^m-(|x|+\epsilon)^{-2}u^p=0,
\end{equation}
with initial condition $u(x,0)=u_0(x)$. Since the coefficient of the zero order term in Eq. \eqref{eq.eps} belongs to $L^{\infty}(\real^N)$, standard theory for quasilinear parabolic equations (see for example \cite[Section 8, Chapter 5]{LSU} and \cite[Sections 3 and 5]{AdB91}, in the latter Eq. \eqref{eq.eps} being studied for $\epsilon=1$ but the same argument holds true with any $\epsilon>0$) gives that there exists a unique solution $u_{\epsilon}$ to the Cauchy problem \eqref{eq.eps}-\eqref{init.cond} and moreover, the comparison principle holds true for Eq. \eqref{eq.eps}. In particular, the solution $u_{\epsilon}$ fulfills the very weak formulation of Eq. \eqref{eq.eps}, which means that, in particular, for any $\varphi\in C_0^{2,1}(\real^N\times(0,T))$ and for any $t_1$, $t_2\in(0,T)$ with $t_1<t_2$ we have
\begin{equation}\label{weak.eps}
\begin{split}
\int_{\real^N}u_{\epsilon}(t_2)\varphi(t_2)\,dx&-\int_{\real^N}u_{\epsilon}(t_1)\varphi(t_1)\,dx-\int_{t_1}^{t_2}\int_{\real^N}u_{\epsilon}(t)\varphi_t(t)\,dx\,dt\\
&-\int_{t_1}^{t_2}\int_{\real^N}u_{\epsilon}^m(t)\Delta\varphi(t)\,dx\,dt=\int_{t_1}^{t_2}\int_{\real^N}\frac{u_{\epsilon}^p(t)}{(|x|+\epsilon)^2}\varphi(t)\,dx\,dt.
\end{split}
\end{equation}
and the initial condition is taken in $L^1$ sense. We furthermore observe that for any $0<\epsilon_1<\epsilon_2<\infty$, the solution $u_{\epsilon_1}$ to Eq. \eqref{eq.eps} with $\epsilon=\epsilon_1$ is a supersolution to Eq. \eqref{eq.eps} with $\epsilon=\epsilon_2$, since
$$
(|x|+\epsilon_1)^{-2}>(|x|+\epsilon_2)^{-2}.
$$
The comparison principle applied for Eq. \eqref{eq.eps} with $\epsilon=\epsilon_2$ then entails that $u_{\epsilon_1}\geq u_{\epsilon_2}$ if $\epsilon_1<\epsilon_2$. We thus obtain a monotone sequence of functions, hence there exists a pointwise limit (which a priori might be equal to $+\infty$)
\begin{equation}\label{interm16}
u(x,t)=\lim\limits_{\epsilon\to0}u_{\epsilon}(x,t).
\end{equation}
Let us now consider the time delay $\tau>0$ given by Lemma \ref{lem.comp} such that $U(x,\tau)\geq u_0(x)$ for any $x\in\real^N$, where $U(x,t)$ is the unique self-similar solution to Eq. \eqref{eq1} given by Theorem \ref{th.uniqSS}. It is immediate to check that $U$ is a supersolution to Eq. \eqref{eq.eps} for any $\epsilon>0$. The comparison principle applied to Eq. \eqref{eq.eps} leads to $u_{\epsilon}(x,t)\leq U(x,t+\tau)$, for any $x\in\real^N$, $t>0$. We then infer that the limit function
$$
u(x,t)=\lim\limits_{\epsilon\to0}u_{\epsilon}(x,t)\leq U(x,t+\tau),
$$
for any $(x,t)\in\real^N\times(0,\infty)$. In particular, we find that $u(x,t)$ is globally defined (in time), and that it belongs to the functional spaces in Definition \ref{def.sol}, as being bounded from above by the function $U(x,t+\tau)$ which belongs to these integral spaces. Moreover, since the integrals in \eqref{weak.eps} are well defined in dimension $N\geq3$ if replacing $u_{\epsilon}$ by $U(\cdot,\cdot+\tau)$, we can pass to the limit in the weak formulation \eqref{weak.eps} by using the Lesbesgue's dominated convergence theorem in order to find that the limit solution $u(x,t)$ defined in \eqref{interm16} satisfies the weak formulation \eqref{weak} corresponding to Eq. \eqref{eq1}. An easy argument based on the monotone convergence theorem also gives that $u_0$ is taken in $L^1$ sense by $u$ as limit when $t\to0$. We can thus conclude that the function $u$ defined in \eqref{interm16} is a weak solution to the Cauchy problem \eqref{eq1}-\eqref{init.cond}.
\end{proof}

\noindent \textbf{Remark.} The weak formulation we are using in this work avoids the problems of having to pass to the limit on the gradients of the approximating functions, as it would have happened if considering the more standard theory where only one integration by parts in the term with the Laplacian is performed (instead of two). We believe that the same function $u$ introduced in \eqref{interm16} will also satisfy this slightly different weak formulation, but its proof requires more work by establishing bounds on the gradients of $u_{\epsilon}$. We refrain from entering this discussion here.

\section*{Extensions, comments and open problems}

The results in this paper represent only a first step towards the general theory of Eq. \eqref{eq1}. Indeed, our aim was not to construct a full theory of well-posedness of Eq. \eqref{eq1}, but the more modest objective of revealing an interesting behavior of a special, unique solution in self-similar form to this equation and then show how an immediate application of it is the existence of at least a global solution for any compactly supported and bounded initial condition. This is already very noticeable, as it is in striking contrast with the cases $m=p=1$ or $m=p>1$ where this existence is limited by the Hardy optimal constant. Of course, there are many new developments that can be considered as further interesting problems related to Eq. \eqref{eq1} and we suggest below some points to be addressed in the future.

\medskip

$\bullet$ \emph{uniqueness of solutions} and comparison principle to Eq. \eqref{eq1}. Recall that in Section \ref{sec.wp} we have constructed the weak solutions as limits of an approximating process. Although this strongly suggests that it is not likely to have a different process leading to different solutions, there is no proof of uniqueness and comparison for solutions to Eq. \eqref{eq1}. This problem is completely non-trivial, since in the neighboring case of letting $\sigma\in(-2,0)$, we have shown recently in \cite{IMS21b} that uniqueness does not hold true by constructing a nonzero solution stemming from zero initial condition, in the form of a forward self-similar solution. In our case, we have the ``feeling" that uniqueness of solutions should hold true, but a proof of it does not seem obvious.

\medskip

$\bullet$ \emph{instantaneous vertical asymptote at $x=0$}. Another open question suggested by the present work is to establish whether a solution $u(x,t)$ with initial condition $u_0$ as in \eqref{init.cond} instantaneously blows up (by forming a vertical asymptote) at $x=0$. In the limiting case $p=m$ with $\sigma=-2$, such an interesting instantaneous blow-up of a wide class of solutions was established by the authors in \cite{IS20b}, together with the interesting point (similar to our case) that this blow-up does not stop the solutions to stay in the suitable functional spaces at any $t>0$. A similar thing is expected to happen here. However, since comparison is not yet available and the analysis of the self-similar profiles does not give a class of ``small" subsolutions to Eq. \eqref{eq1}, estimates pushing $u$ to increase very rapidly near $x=0$ are to be obtained.

\medskip

$\bullet$ \emph{large time behavior of solutions}. This is a very natural question once we have a very strong candidate for the large time behavior, in the form of a unique self-similar solution with some specific behavior. One can readily see that Lemma \ref{lem.comp} also provides, in terms of self-similar variables, an optimal upper bound for the general solutions. We then miss a good lower bound that has to be obtained in a different way, probably using energy-like estimates and functional inequalities, in order to establish that the self-similar scale is the right one for the dynamics of the equation. Moreover, a step of establishing compactness (or Holder estimates for the solutions) is required for convergence.

\bigskip

\noindent \textbf{Acknowledgements} R. I. and A. S. are partially supported by the Spanish project PID2020-115273GB-I00.

\bibliographystyle{plain}

\begin{thebibliography}{1}

\bibitem{AP00}
J. A. Aguilar Crespo and I. Peral Alonso, \emph{Global behaviour of the Cauchy problem for some critical nonlinear parabolic equations}, SIAM J. Math. Anal., \textbf{31} (2000) 1270-1294.

\bibitem{ASZ01}
J. R. Anderson, S. Ning and H. Zhang, \emph{Existence and uniqueness of solutions of degenerate parabolic equations in exterior domains}, Nonlinear Anal., \textbf{44} (2001), no. 4, 453-468.

\bibitem{AdB91}
D. Andreucci and E. DiBenedetto, \emph{On the Cauchy problem and initial traces for a class of evolution equations with strongly nonlinear sources}, Ann. Scuola Norm. Sup. Pisa, \textbf{18} (1991), no. 3, 363-441.

\bibitem{BG84}
P. Baras and J. Goldstein, \emph{The heat equation with a singular potential}, Trans. Amer. Math. Soc., \textbf{284} (1984), no. 1, 121-139.

\bibitem{BS19}
B. Ben Slimene, \emph{Asymptotically self-similar global solutions for Hardy-H\'enon parabolic systems}, Differ. Equ. Appl., \textbf{11} (2019), no. 4, 439-462.

\bibitem{BSTW17}
B. Ben Slimene, S. Tayachi and F. B. Weissler, \emph{Well-posedness, global existence and large time behavior for Hardy-H\'enon parabolic equations}, Nonlinear Anal., \textbf{152} (2017), 116-148.

\bibitem{CM99}
X. Cabr\'e and Y. Martel, \emph{Existence versus explosion instantan\'ee por des \'equations de la chaleur lin\'eaires avec potentiel singulier}, C. R. Acad. Sci. Paris, \textbf{329} (1999), no. 11, 973-978.

\bibitem{Carr}
J. Carr, \emph{Applications of Centre Manifold Theory}, Springer Verlag, New York, 1981.

\bibitem{CIT21a}
N. Chikami, M. Ikeda and K. Taniguchi, \emph{Well-posedness and global dynamics for the critical Hardy-Sobolev parabolic equation}, Preprint ArXiv no. 2009.07108

\bibitem{CIT21b}
N. Chikami, M. Ikeda and K. Taniguchi, \emph{Optimal well-posedness and forward self-similar solution for the Hardy-H\'enon parabolic equation in critical weighted Lebesgue spaces}, Preprint ArXiv no. 2104.14166

%

\bibitem{FT00}
S. Filippas and A. Tertikas, \emph{On similarity solutions of a heat equation with a nonhomogeneous nonlinearity}, J. Differential Equations, \textbf{165} (2000), no. 2, 468-492.

\bibitem{FK80}
A. Friedman and S. Kamin, \emph{The asymptotic behavior of gas in an $n$-dimensional porous medium}, Trans. Amer. Math. Soc., \textbf{262} (1980), no. 2, 551-563.

\bibitem{Fu66}
H. Fujita, \emph{On the blowing up of solutions of the Cauchy problem for $u_t=\Delta u+u^{1+\alpha}$}, J. Fac. Sci. Univ. Tokyo Sect. I, \textbf{13} (1966), 109-124.

\bibitem{GK03}
J. A. Goldstein and I. Kombe, \emph{Nonlinear degenerate prabolic equations with singular lower-order term}, Adv. Differential Equations, \textbf{8} (2003), no. 10, 1153-1192.

\bibitem{GGK05}
G. R. Goldstein, J. A. Goldstein and I. Kombe, \emph{Nonlinear parabolic equations with singular coefficient and critical exponent}, Appl. Anal., \textbf{84} (2005), no. 6, 571-583.

\bibitem{GZ02}
J. A. Goldstein and Qi S. Zhang, \emph{Linear parabolic equations with strong singular potentials}, Trans. Amer. Math. Society, \textbf{355} (2002), no. 1, 197-211.



\bibitem{HS21}
K. Hisa and M. Sierzega, \emph{Existence and nonexistence of solutions to the Hardy parabolic equation}, Preprint ArXiv no. 2102.04079

\bibitem{HT21}
K. Hisa and J. Takahashi, \emph{Optimal singularities of initial data for solvability of the Hardy parabolic equation}, J. Differential Equations, \textbf{296} (2021), 822-848.

\bibitem{IL13}
R. G. Iagar and Ph.~Lauren\ced{c}ot, \emph{Existence and uniqueness of very singular solutions for a fast diffusion equation with gradient absorption},  J. London Math. Soc., \textbf{87} (2013), 509-529.

\bibitem{ILS22}
R. G. Iagar, Ph. Lauren\ced{c}ot and A. S\'anchez, \emph{Self-similar shrinking of supports and non-extinction for a nonlinear diffusion equation with strong nonhomogeneous absorption}, Submitted (2022), Preprint ArXiv no. 2204.09307.

\bibitem{IMS22}
R. G. Iagar, A. I. Mu\~{n}oz and A. S\'anchez, \emph{Self-similar blow-up patterns for a reaction-diffusion equation with weighted reaction in general dimension}, Comm. Pure Appl. Anal., \textbf{21} (2022), no. 3, 891-925.

\bibitem{IMS21b}
R. G. Iagar, A. I. Mu\~{n}oz and A. S\'anchez, \emph{Self-similar solutions preventing finite time blow-up for reaction-diffusion equations with singular potential}, Submitted (2021), Preprint ArXiv no. 2111.04806.

\bibitem{IS19}
R. G. Iagar and A. S\'anchez, \emph{Blow up profiles for a quasilinear reaction-diffusion equation with weighted reaction with linear growth}, J. Dynam. Differential Equations, \textbf{31} (2019), no. 4, 2061-2094.

\bibitem{IS20b}
R. G. Iagar and A. S\'anchez, \emph{Instantaneous and finite time blow-up of solutions to a reaction-diffusion equation with Hardy-type singular potential}, J. Math. Anal. Appl., \textbf{491} (2020), no. 1, paper no. 124244, 11 pages.

\bibitem{IS21a}
R. G. Iagar and A. S\'anchez, \emph{Blow up profiles for a quasilinear reaction-diffusion equation with weighted reaction}, J. Differential Equations, \textbf{272} (2021), no. 1, 560-605.

\bibitem{IS21b}
R. G. Iagar and A. S\'anchez, \emph{Self-similar blow-up profiles for a reaction-diffusion equation with critically strong weighted reaction}, J. Dynam. Differential Equations, to appear, online (2021), DOI https://doi.org/10.1007/s10884-020-09920-w.

\bibitem{IS22}
R. G. Iagar and A. S\'anchez, \emph{Separate variable blow-up patterns for a reaction-diffusion equation with critical weighted reaction}, Nonlinear Anal., \textbf{217} (2022), article no. 112740.


\bibitem{IV10}
R. G. Iagar and J. L. V\'azquez, \emph{Anomalous large-time behaviour of the $p$-Laplacian flow in an exterior domain in low dimension}, J. Eur. Math. Soc., \textbf{12} (2010), no. 1, 249-277.

\bibitem{Ko04}
I. Kombe, \emph{Doubly nonlinear parabolic equations with singular lower order term}, Nonlinear Anal., \textbf{56} (2004), no. 2, 185-199.

\bibitem{LSU}
O. A. Ladyzhenskaya, V. A. Solonnikov and N. N. Uraltseva, \emph{Linear and quasi-linear equations of parabolic type}, Translations of Mathematical Monographs, Volume 23, American Mathematical Society.

\bibitem{MS21}
A. Mukai and Y. Seki, \emph{Refined construction of Type II blow-up solutions for semilinear heat equations with Joseph-Lundgren supercritical nonlinearity}, Discrete Cont. Dynamical Systems, \textbf{41} (2021), no. 10, 4847-4885.

\bibitem{Pe}
L. Perko, \emph{Differential equations and dynamical systems. Third edition}, Texts in Applied Mathematics, \textbf{7}, Springer Verlag, New York, 2001.

\bibitem{Qi98}
Y.-W. Qi, \emph{The critical exponents of parabolic equations and blow-up in $\real^n$}, Proc. Royal Soc. Edinburgh A, \textbf{128} (1998), 123-136.


%


\bibitem{Su02}
R. Suzuki, \emph{Existence and nonexistence of global solutions of quasilinear parabolic equations}, J. Math. Soc. Japan, \textbf{54} (2002), no. 4, 747-792.

\bibitem{T20}
S. Tayachi, \emph{Uniqueness and non-uniqueness of solutions for critical Hardy-H\'enon parabolic equations}, J. Math. Anal. Appl. \textbf{488} (2020), no. 1, paper no. 123976, 51 pp.

\bibitem{VPME}
J. L. V\'azquez, \emph{The porous medium equation. Mathematical theory}, Oxford Monographs in Mathematics, Oxford University Press, 2007.

\bibitem{VZ00}
J. L. V\'azquez and E. Zuazua, \emph{The Hardy inequality and the asymptotic behavior of the heat equation with an inverse square potential}, J. Functional Analysis, \textbf{173} (2000), no. 1, 103-153.

\bibitem{VZ12}
J. L. V\'azquez and N. Zographoulos, \emph{Functional aspects of the Hardy inequality: appearance of a Hidden energy}, J. Evol. Equ., \textbf{12} (2012), no. 3, 713-739.


\bibitem{YeYin}
H. Ye and J. Yin, \emph{Uniqueness of self-similar very singular solution for non-Newtonian polytropic filtration equations with gradient absorption}, Electronic J. Differential Equations, \textbf{2015} (2015), no. 83, 1-9.

\end{thebibliography}

\end{document}